\documentclass{article}
\usepackage{amsfonts}
\usepackage{amsmath}

\newtheorem{theorem}{Theorem}

\newtheorem{corollary}[theorem]{Corollary}

\newtheorem{definition}[theorem]{Definition}
\newtheorem{example}[theorem]{Example}

\newtheorem{proposition}[theorem]{Proposition}
\newtheorem{remark}[theorem]{Remark}

\input{tcilatex}
\begin{document}

\title{Applications of derivative and difference operators on some sequences}
\author{Ayhan Dil \\
Department of Mathematics, Akdeniz University, 07058 Antalya Turkey\\
adil@akdeniz.edu.tr \and Erkan Muniro\u{g}lu \\
Department of Finance (Public Economics), Ankara University, 06590 Ankara
Turkey\\
emuniroglu@ankara.edu.tr}
\maketitle

\begin{abstract}
In this study, depending on the upper and the lower indices of the
hyperharmonic number $h_{n}^{(r)}$, nonlinear recurrence relations are
obtained. It is shown that generalized harmonic number and hyperharmonic
number can be obtained from derivatives of the binomial coefficients. Taking
into account of difference and derivative operators, several identities of
the harmonic and hyperharmonic numbers are given. Negative-ordered
hyperharmonic number is defined and its alternative representations are
given.


\end{abstract}

\section{\textbf{INTRODUCTION}}

Harmonic numbers are longstanding subject of study and they are significant
in various branches of analysis and number theory. The $n$-th harmonic
number is the $n$-th partial sum of the harmonic series and defined by:%
\begin{equation*}
H_{n}:=\sum_{k=1}^{n}\frac{1}{k},\text{ \ \ }\left( n\in\mathbb{N}%
:=\left\{1,2,3,\ldots\right\} \right) ,  \label{1}
\end{equation*}
where the empty sum $H_{0}$ is conventionally understood to be zero. As a
matter of fact, various generalizations of this number are also widely
studied. Among many other generalizations we are interested in two famous
generalizations of these numbers, namely generalized harmonic number and
hyperharmonic number.

For a positive integer $n$ and an integer $m$ the $n$-th generalized
harmonic number of order $m$ is defined by%
\begin{equation*}
H_{n}^{\left( m\right) }:=\sum_{k=1}^{n}\frac{1}{k^{m}}.
\end{equation*}
It is convenient to set $H_{n}^{\left( m\right) }=0$ for $n\leq0$. Hence for 
$m>1$, $H_{n}^{\left( m\right) }$ is the $n$-th partial sum of the Riemann
zeta function $\zeta\left( m\right) .$

Hyperharmonic number are another important generalization of harmonic
number. For $r\in\mathbb{N},$ the $n$-th hyperharmonic number of order $r$
is defined by \cite{CG, K}%
\begin{equation}
h_{n}^{(r)}=\sum_{k=1}^{n}h_{k}^{(r-1)},\text{ \ \ \ \ \ }h_{n}^{(1)}:=H_{n}
.  \label{2}
\end{equation}
Here we assume that $h_{n}^{(0)}:=1/n$, $\left( n\geq1\right) $ and $%
h_{0}^{(r)}:=0$, $\left( r\geq0\right) $. Hyperharmonic numbers are closely
related to multiple zeta functions (see \cite{BBG, BG, DB, DMC, MD2}),
discrete mathematics and combinatorial analysis (see \cite{BGG, C, CG, DM}).

These numbers have an expression in terms of binomial coefficients and
harmonic numbers \cite{CG, K, MD2}:%
\begin{equation}
h_{n}^{(r)}=\binom{n+r-1}{r-1}(H_{n+r-1}-H_{r-1}).  \label{3}
\end{equation}
Also we have the following representation \cite{BGG, DM}:%
\begin{equation*}
h_{n}^{(r)}=\sum_{k=1}^{n}\binom{n+r-k-1}{r-1}\frac{1}{k}.  \label{4}
\end{equation*}
In Subsection 2.1, recurrence relations for both upper and lower indices of
hyperharmonic numbers are given. Properties of the coefficients of these
relations are examined.

For a $n$-times differentiable function $f\left( x\right) $ define the
derivative operator $D_{x}$ by%
\begin{equation*}
D_{x}^{n}f\left( x\right) =\frac{d^{n}}{dx^{n}}f\left( x\right) .
\end{equation*}
In Subsection 2.2 we focus on the following simple equation \cite{HWG, PS}:

\begin{equation}
D_{x}\binom{x+n}{n}\mid_{x=0}=H_{n}  \label{10}
\end{equation}
which provides a connection between analysis and combinatorics. Peter and
Schneider used (\ref{10}) in \cite{PS} as a starting point, and they dealt
with some important harmonic number identities. Also\ Chu and Donno \cite{CD}
used the classical hypergeometric summation theorems to derive numerious
identities involving harmonic numbers. In the paper \cite{SA}, Sofo used the
idea of the consecutive derivative operator of binomial coefficients to give
integral representation for series containing binomial coefficients and
harmonic numbers. Choi \cite{C} showed how one can obtain identities about
certain finite series involving binomial coefficients, harmonic numbers and
generalized harmonic numbers by applying the usual differential operator to
a known identity. In \cite{Y}, in terms of the telescoping method, Yan and
Liu constructed a binomial identity first; then by applying the derivative
operator to that identity, authors derived numerous interesting harmonic
number identities.

In this subsection, we first obtain generalization of (\ref{10}) for
generalized harmonic number%
\begin{equation*}
D_{x}\binom{x+n^{m}}{n^{m}}_{m}\mid_{x=0}=H_{n}^{\left( m\right) }.
\end{equation*}
For this generalization, the concept of ``leaping binomial coefficients"%
\begin{equation*}
\binom{x+n^{m}}{n^{m}}_{m}=\frac{1}{\left( n!\right) ^{m}}\dprod
\limits_{i=1}^{n}\left( x+i^{m}\right)
\end{equation*}
are defined and the relationship between the leaping binomial coefficients
and the classical binomial coefficients are obtained. After that, the
generalization of (\ref{10}) for hyperharmonic number%
\begin{equation*}
D_{x}\binom{x+n+r-1}{n}\mid_{x=0}=h_{n}^{\left( r\right) }
\end{equation*}
is also obtained. Using this generalization, dozens of new formulas
containing harmonic, hyperharmonic and generalized harmonic numbers are
given in the light of H. W. Gould's book \cite{HWG} about binomial
coefficients.

The difference operator%
\begin{equation}
\Delta f\left( x\right) =f\left( x+1\right) -f\left( x\right) \text{ and }%
\Delta^{0}f\left( x\right) =f\left( x\right)  \label{fop}
\end{equation}
has a wide range of applications, in particular in discrete mathematics and
differential equations theory (see \cite{AS, KNB, Slo}). This is the finite
analog of the derivative operator \cite{GKP}. In Subsections 2.3 and 2.4,
the difference operator is primarily used to investigate the properties of
harmonic and hyperharmonic numbers. Meanwhile, negative-ordered
hyperharmonic numbers are presented and various representations are obtained.

In Subsection 2.5, Fibonacci numbers%
\begin{equation*}
F_{n}=F_{n-1}+F_{n-2}\text{ with }F_{0}:=0,F_{1}:=1
\end{equation*}
are examined with the help of difference operator. We also encounter with
the Fibonacci number of negative indexed and give representation of these
numbers.

In Subsection 2.6, results for hyperbolic functions are given to illustrate
the prevalence of the difference operator in practice.

Finally, in the Appendix the reader can find two tables of summation
formulas related to harmonic, hyperharmonic and generalized harmonic
numbers. These formulas are applications of the results that we obtained in
Subsection 2.2.

\section{\textbf{MAIN RESULTS}}

\subsection{\textbf{A symmetric identity and nonlinear first-order
recurrences for hyperharmonic numbers}}

The following proposition shows a balance between the upper and the lower
indices of hyperharmonic numbers.

\begin{proposition}
We have 
\begin{equation}
\sum_{s=1}^{n}h_{s}^{(r-1)}-\sum_{k=1}^{r}h_{n-1}^{(k)}=\frac{1}{n}
=h_{n}^{\left( 0\right) }.  \label{5}
\end{equation}
\end{proposition}

\begin{proof}
	From (\ref{2}) we obtain%
	\begin{equation}
	h_{n}^{(r-1)}=h_{n}^{(r)}-h_{n-1}^{(r)}.\label{6}%
	\end{equation}
	Telescoping sum on $r$ of (\ref{6}) gives the statement.
\end{proof}

\begin{remark}
We can write the result (\ref{5}) in the equivalent form 
\begin{equation*}
h_{n}^{(r)}=\sum_{k=1}^{r}h_{n-1}^{(k)}+\frac{1}{n}
\end{equation*}
which is already obtained in \cite{BGG}. In the light of this equation we
get the following general result 
\begin{equation*}
h_{n}^{(r+s)}-h_{n}^{(s)}=\sum_{k=1}^{r}h_{n-1}^{(k+s)},
\end{equation*}
where $s\in\mathbb{N}$.
\end{remark}

Since the hyperharmonic number of order $r$ is an $r$-fold sum, some
calculations involving these numbers can be difficult. For this reason,
various studies have been performed in which some other representations of
hyperharmonic numbers are given and recurrences are obtained. Now we give
recurrence relations for the lower and the upper indices of $h_{n}^{(r)}.$

\begin{proposition}
\label{rr}A recurrence with respect to the lower index $n$ is 
\begin{equation}
h_{n}^{(r+1)}=\alpha h_{n-1}^{(r+1)}+\beta,  \label{8}
\end{equation}
and a recurrence with respect to the upper index $r$ is 
\begin{equation}
\left( \alpha-1\right) h_{n}^{(r+1)}=\alpha h_{n}^{(r)}-\beta,  \label{9}
\end{equation}
where $\alpha=\alpha\left( n,r\right) =1+\frac{r}{n}$ and $\beta
=\beta\left( n,r\right) =\frac{1}{n+r}\binom{n+r}{r}$.
\end{proposition}

\begin{proof}
	Using the fact that $H_{n+r}=H_{n+r-1}+\frac{1}{n+r}$ and (\ref{3})\ it
	follows that%
	\[
	h_{n}^{(r+1)}=\frac{1}{n+r}\binom{n+r}{r}+\binom{n+r}{r}(H_{n+r-1}-H_{r}).
	\]
	Considering the summation identity $\binom{n}{r-1}+\binom{n}{r}=\binom{n+1}%
	{r}$ we get%
	\[
	h_{n}^{(r+1)}=\frac{1}{n+r}\binom{n+r}{r}+h_{n-1}^{(r+1)}+\binom{n+r-1}%
	{r-1}(H_{n+r-1}-H_{r}).
	\]
	Now, the following identity combines with (\ref{3}) to give (\ref{8}):%
	\[
	\binom{n+r-1}{r-1}=\frac{r}{n}\binom{n+r-1}{r}.
	\]

	For the second recurrence we start with the equation%
	\begin{equation}
	h_{n}^{(r+1)}=h_{n}^{(r)}+h_{n-1}^{(r+1)}.\label{y1}%
	\end{equation}
	On the other hand (\ref{8}) gives%
	\begin{equation}
	h_{n-1}^{(r+1)}=\frac{1}{\alpha}h_{n}^{(r+1)}-\frac{\beta}{\alpha}.\label{y2}%
	\end{equation}
	Considering (\ref{y2}) in (\ref{y1}) gives%
	\[
	h_{n}^{(r+1)}=h_{n}^{(r)}+\frac{1}{\alpha}h_{n}^{(r+1)}-\frac{\beta}{\alpha},
	\]
	which can equally well be written%
	\[
	h_{n}^{(r+1)}=\frac{\alpha}{\alpha-1}h_{n}^{(r)}-\frac{\beta}{\alpha-1}.
	\]
	
\end{proof}

\begin{example}
Let us show the usefulness of the Proposition \ref{rr}. Fixing $n=2$ in (\ref%
{8}) gives 
\begin{equation*}
h_{2}^{(r+1)}=\left( 1+\frac{r}{2}\right) h_{1}^{(r+1)}+\frac{r+1}{2}.
\end{equation*}
Remembering that $h_{1}^{(r)}=1$ for any order $r,$ we get a general formula
for $h_{2}^{(r)}$ as 
\begin{equation*}
h_{2}^{(r+1)}=r+1+\frac{1}{2}.
\end{equation*}
On the other hand if we fix $r=1$ in (\ref{9}), then we get a general
formula for $h_{n}^{(2)}$ in terms of $H_{n}$ as
\end{example}

\begin{equation*}
h_{n}^{(2)}=\left( n+1\right) H_{n}-n.
\end{equation*}

\begin{remark}
Let us observe 
\begin{equation*}
\alpha\left( n,r\right) =1+\frac{r}{n}=\frac{r}{n}\alpha\left( r,n\right)
\end{equation*}
and 
\begin{equation*}
\beta\left( n,r\right) =\frac{1}{n+r}\binom{n+r}{r}=\beta\left( r,n\right) .
\end{equation*}
With the help of the equality 
\begin{equation*}
\beta\left( n,r\right) =\frac{1}{n+r}\binom{n+r}{r}=\frac{1}{r}\binom {n+r-1%
}{n},
\end{equation*}
we have the ordinary generating function of $\beta\left( k,r\right) $ 
\begin{equation*}
\dsum \limits_{k=0}^{\infty} \beta\left( k,r\right) x^{k}=\frac{1}{r} \dsum
\limits_{k=0}^{\infty} \binom{k+r-1}{k}x^{k}=\frac{1}{r\left( 1-x\right) ^{r}%
},
\end{equation*}
and also the ordinary generating function of $\alpha\left( k,r\right) $ 
\begin{equation*}
\dsum \limits_{k=1}^{\infty} \alpha\left( k,r\right) x^{k}=\frac{x}{1-x}%
-r\ln\left( 1-x\right) .
\end{equation*}
\end{remark}

Next proposition enables us to get a closed form evaluation for the finite
sum of $\beta\left( k,r\right) $ in terms of $\alpha$ and $\beta$. Proof of
it can be directly seen from the following basic properties of the binomal
coefficients \cite[p. 174]{GKP}:%
\begin{equation}
\binom{k+r}{r}=\frac{k+r}{r}\binom{k+r-1}{r-1}  \label{b1}
\end{equation}
and%
\begin{equation}
\dsum \limits_{k=0}^{n}\binom{k+r-1}{k}=\binom{n+r}{n}.  \label{b2}
\end{equation}

\begin{proposition}
\label{bt} Let $n,r\in\mathbb{N}$. Then 
\begin{equation*}
\dsum \limits_{k=0}^{n} \beta\left( k,r\right) =\frac{\alpha\left(
n,r\right) \beta\left( n,r\right) }{\alpha\left( n,r\right) -1}
\end{equation*}
or equally 
\begin{equation*}
\dsum \limits_{k=0}^{n} \frac{1}{k+r}\binom{k+r}{r}=\frac{1}{r}\binom{n+r}{n}%
.  \label{bs}
\end{equation*}
\end{proposition}

It is possible to generalize Proposition \ref{bt} by considering the concept
of "falling factorial". Recall that falling factorial is defined with the
equation \cite{GKP}:%
\begin{equation*}
x^{\underline{n}}=x\left( x-1\right) \left( x-2\right) ...(x-n+1).
\end{equation*}

\begin{proposition}
Let $m,r\in\mathbb{N}$, then 
\begin{equation*}
\dsum \limits_{k=0}^{n} \frac{1}{\left( k+r\right) ^{\underline{m}}}\binom{%
k+r}{r}=\frac {1}{r^{\underline{m}}}\binom{n+r-m+1}{n}.
\end{equation*}
\end{proposition}

\begin{proof}
	Because of (\ref{b1}) we have%
	\[
	\frac{1}{k+r}\binom{k+r}{r}=\frac{1}{r}\binom{k+r-1}{r-1}%
	\]
	and also%
	\[
	\frac{1}{\left(  k+r\right)  ^{\underline{2}}}\binom{k+r}{r}=\frac
	{1}{r^{\underline{2}}}\binom{k+r-2}{r-2}.
	\]
	Hence, in general we can write%
	\[
	\frac{1}{\left(  k+r\right)  ^{\underline{m}}}\binom{k+r}{r}=\frac
	{1}{r^{\underline{m}}}\binom{k+r-m}{r-m}.
	\]
	Summing both sides as%
	\[%
	{\displaystyle\sum\limits_{k=0}^{n}}
	\frac{1}{\left(  k+r\right)  ^{\underline{m}}}\binom{k+r}{r}=\frac
	{1}{r^{\underline{m}}}%
	{\displaystyle\sum\limits_{k=0}^{n}}
	\binom{k+r-m}{r-m}%
	\]
	and employing (\ref{b2}) give the statement.
\end{proof}

\subsection{\textbf{Hyperharmonic and generalized harmonic numbers via
derivative operator}}

Here we generalize the identity (\ref{10}) both for generalized harmonic
number and hyperharmonic number.

The following definition plays a key role in generalizing (\ref{10}) for
generalized harmonic number.

\begin{definition}
For any parameter $x$ and positive integers $m$ and $n$ ``leaping binomial
coefficients" are defined by 
\begin{align*}
\binom{x+n^{m}}{n^{m}}_{m} & =\frac{\left( x+n^{m}\right) \left( x+\left(
n-1\right) ^{m}\right) \cdots\left( x+2^{m}\right) \left( x+1^{m}\right) }{%
\left( n!\right) ^{m}} \\
& =\frac{1}{\left( n!\right) ^{m}} \dprod \limits_{i=1}^{n} \left(
x+i^{m}\right).
\end{align*}
\end{definition}

Now we are ready to give a generalization of (\ref{10}) for generalized
harmonic number.

\begin{proposition}
For any $m$, $n\in\mathbb{N}$ we have 
\begin{equation*}
D_{x}\binom{x+n^{m}}{n^{m}}_{m}\mid_{x=0}=H_{n}^{\left( m\right) }.
\label{Dgh}
\end{equation*}
\end{proposition}

\begin{proof}
	Let us observe the following equation%
	\begin{align*}
	D_{x}\binom{x+n^{m}}{n^{m}}_{m}  & =\frac{\left(  x+n^{m}\right)  \left(
		x+\left(  n-1\right)  ^{m}\right)  \cdots\left(  x+2^{m}\right)  \left(
		x+1^{m}\right)  }{\left(  n!\right)  ^{m}}\times\\
	& \left\{  \frac{1}{\left(  x+n^{m}\right)  }+\frac{1}{\left(  x+\left(
		n-1\right)  ^{m}\right)  }+\cdots+\frac{1}{\left(  x+1^{m}\right)  }\right\}
	.
	\end{align*}
	Here evaluating both sides at $x=0$ gives the statement.
\end{proof}

Relation between the classical binomial coefficients and the leaping
binomial coefficients is given by the following proposition.

\begin{proposition}
For any positive integers $m$ and $n\geq2$ we have 
\begin{equation*}
\binom{x+n^{m}}{n^{m}}_{m}=\frac{\frac{\left( n^{m}\right) !}{\left(
n!\right) ^{m}}\binom{x+n^{m}}{n^{m}}}{ \dprod \limits_{i=2}^{n} \binom{%
x+i^{m}-1}{i^{m}-\left( i-1\right) ^{m}-1}\left( i^{m}-\left( i-1\right)
^{m}-1\right) !}.
\end{equation*}
\end{proposition}

\begin{proof}
	We start with the definition of the leaping binomial coefficients, and replace
	the terms%
	\[
	\left(  x+i^{m}\right)  \left(  x+\left(  i-1\right)  ^{m}\right)
	\]
	by%
	\[
	\frac{\left(  x+i^{m}\right)  \left(  x+i^{m}-1\right)  \left(  x+i^{m}%
		-2\right)  \cdots\left(  x+\left(  i-1\right)  ^{m}+1\right)  \left(
		x+\left(  i-1\right)  ^{m}\right)  }{\left(  x+i^{m}-1\right)  \left(
		x+i^{m}-2\right)  \cdots\left(  x+\left(  i-1\right)  ^{m}+1\right)  }%
	\]
	for $i=2,3,...,n$. Reorganizing these terms we get%
	\begin{align*}
	\binom{x+n^{m}}{n^{m}}_{m}  & =\frac{\left(  n^{m}\right)  !}{\left(
		n!\right)  ^{m}\left(  n^{m}-\left(  n-1\right)  ^{m}-1\right)  !...\left(
		2^{m}-1^{m}-1\right)  !}\times\\
	& \frac{\binom{x+n^{m}}{n^{m}}}{\binom{x+n^{m}-1}{n^{m}-\left(  n-1\right)
			^{m}-1}\binom{x+\left(  n-1\right)  ^{m}-1}{\left(  n-1\right)  ^{m}-\left(
			n-2\right)  ^{m}-1}...\binom{x+2^{m}-1}{2^{m}-1^{m}-1}},
	\end{align*}
	which completes the proof.
\end{proof}

Some preparation is needed to generalize (\ref{10}) for hyperharmonic
number. With the help of the operator $D_{x}$ and the classical gamma
function%
\begin{equation*}
\Gamma(z)=\dint \limits_{0}^{\infty}e^{-t}t^{z-1}dt,\text{ \ \ \ \ }\func{Re}%
\left( z\right) >0,
\end{equation*}
the digamma function is defined by \cite{AS}:%
\begin{equation*}
\psi\left( z\right) =D_{z}\log\left( \Gamma(z)\right) =\frac {%
\Gamma^{\prime}(z)}{\Gamma(z)},\text{ \ \ }\left( z\in\mathbb{C}%
\backslash\left\{ 0,-1,-2,-3,\ldots\right\} \right) .
\end{equation*}

The rising factorial is defined by%
\begin{equation*}
z\overline{^{n}}=z\left( z+1\right) \left( z+2\right) ...(z+n-1).
\end{equation*}
The rising factorial $z\overline{^{n}}$ is sometimes denoted by $\left(
z\right) _{n}$ \cite{AS, GKP}. It is closely related to the Euler's gamma
function by the relation%
\begin{equation*}
z\overline{^{n}}=\frac{\Gamma\left( z+n\right) }{\Gamma\left( z\right) }.
\end{equation*}

Derivative of $z^{\overline{n}}$ turns out to be%
\begin{equation}
D_{z}z^{\overline{n}}=z^{\overline{n}}\left( \psi\left( z+n\right)
-\psi\left( z\right) \right) .  \label{pd}
\end{equation}
Actually one can easily prove (\ref{pd}) by considering the equation%
\begin{equation*}
\psi\left( z+1\right) =\psi\left( z\right) +\frac{1}{z}.  \label{recgam}
\end{equation*}

Evaluating $\psi$ at positive integers gives%
\begin{equation}
\psi\left( n\right) =H_{n-1}-\gamma.  \label{ph}
\end{equation}
We recall that $\gamma=-\psi\left( 1\right) $\ is Euler-Mascheroni constant.

Considering (\ref{3}) and (\ref{ph}) we also have generalization of (\ref{ph}%
) for hyperharmonic number as \cite{D, M}%
\begin{equation}
h_{n}^{\left( r\right) }=\frac{r^{\overline{n}}}{n!}\left( \psi\left(
n+r\right) -\psi\left( r\right) \right) .  \label{phh}
\end{equation}

Now we are ready to give a generalization of (\ref{10}) for hyperharmonic
number.

\begin{proposition}
Let $n\in\mathbb{N\cup}\left\{ 0\right\} $ and $r\in\mathbb{N}$. Then 
\begin{equation}
D_{x}\binom{x+n+r-1}{n}\mid_{x=0}=h_{n}^{\left( r\right) }.  \label{11}
\end{equation}
\end{proposition}

\begin{proof}
	In the light of (\ref{pd}) we write%
	\[
	D_{x}\binom{x+n+r-1}{n}=D_{x}\frac{\left(  x+r\right)  ^{\overline{n}}}%
	{n!}=\frac{\left(  x+r\right)  ^{\overline{n}}}{n!}\left(  \psi\left(
	x+n+r\right)  -\psi\left(  x+r\right)  \right).
	\]
	Setting $x=0$ in the above equation and considering (\ref{phh})\ we get the
	desired result.
\end{proof}

\begin{remark}
Let us consider the binomial theorem 
\begin{equation*}
\sum_{n=0}^{\infty}\binom{x+n}{n}z^{n}=\frac{1}{\left( 1-z\right) ^{x+1}}
\end{equation*}
and in general 
\begin{equation*}
\sum_{n=0}^{\infty}\binom{x+n+r-1}{n}z^{n}=\frac{1}{\left( 1-z\right) ^{x+r}}%
,
\end{equation*}
where $\left\vert z\right\vert <1$ and for any $x\in\mathbb{C}$, $r\in%
\mathbb{N}$. Hence the generating functions of harmonic and hyperharmonic
numbers are direct consequences of (\ref{11}): 
\begin{equation*}
\sum_{n=0}^{\infty}H_{n}z^{n}=-\frac{\ln\left( 1-z\right) }{1-z},
\end{equation*}
\begin{equation*}
\sum_{n=0}^{\infty}h_{n}^{\left( r\right) }z^{n}=-\frac{\ln\left( 1-z\right) 
}{\left( 1-z\right) ^{r}}.
\end{equation*}
\end{remark}

\begin{remark}
Considering (\ref{11}) with the following binomial equation \cite[p. 174]%
{GKP} 
\begin{equation*}
\dsum \limits_{k=0}^{n} \binom{x+k+r-2}{k}=\binom{x+n+r-1}{n}
\end{equation*}
gives (\ref{2}).
\end{remark}

\begin{remark}
Let us recall the following binomial equation (see \cite[p. 174]{GKP}): 
\begin{equation}
\dsum \limits_{j=0}^{n} \binom{x+j+r-1}{j}=\left( 1+\frac{n}{x+r}\right) 
\binom{x+n+r-1} {n}.  \label{13}
\end{equation}
Applying $D_{x}$ both sides of (\ref{13}) and remembering that 
\begin{equation*}
\binom{n+r}{r}=\frac{n+r}{r}\binom{n+r-1}{r-1},
\end{equation*}
we get 
\begin{equation*}
\left( \alpha-1\right) h_{n-1}^{\left( r+1\right) }=h_{n}^{\left( r\right)
}-\beta
\end{equation*}
which is an alternative proof of (\ref{9}).
\end{remark}

\subsection{\textbf{Harmonic numbers via difference operator}}

The difference operator $\Delta$ (see (\ref{fop})) is a linear operator,
hence%
\begin{equation*}
\Delta\left[ af\left( x\right) +bg\left( x\right) \right] =a\Delta f\left(
x\right) +b\Delta g\left( x\right)
\end{equation*}
holds for any constants $a$ and $b$. Applying $\Delta$ operator $n-$times to
a suitable function $f$, we get \cite{AS, KNB, GKP, QG}:%
\begin{equation}
\Delta^{n}f\left( x\right) =\dsum \limits_{i=0}^{n}\left( -1\right) ^{n-i}%
\binom{n}{i}f\left( x+i\right) .  \label{gfark}
\end{equation}

\begin{proposition}
\label{teo4} We have 
\begin{equation*}
\Delta^{n}\left( xf\left( x\right) \right) =x\Delta^{n}f\left( x\right)
+n\Delta^{n-1}f\left( x+1\right) .
\end{equation*}
\end{proposition}

\begin{proof}
	By induction on $n$.
\end{proof}

Now, we use this result as a tool to investigate the properties of harmonic
numbers.

\begin{proposition}
\label{son1}For any $k\in\mathbb{N}$ we have 
\begin{equation}
\sum_{i=0}^{k}\left( -1\right) ^{i+1}\binom{k}{i}H_{n+i}=\frac{\left(
k-1\right) !}{\left( n+1\right) ^{\overline{k}}}.  \label{bphg}
\end{equation}
\end{proposition}

\begin{proof}
	Setting $f\left(  n\right)  =H_{n}$ in (\ref{gfark}) gives%
	\[
	\Delta^{k}H_{n}=%
	{\displaystyle\sum\limits_{i=0}^{k}}
	\left(  -1\right)  ^{k-i}\binom{k}{i}H_{n+i}.
	\]
	On the other hand considering $H_{n}$ with difference operator gives%
	\[
	\Delta^{k}H_{n}=\frac{\left(  -1\right)  ^{k+1}\left(  k-1\right)  !}{\left(
		n+1\right)  ^{\overline{k}}}.
	\]
	Equality of these two equations completes the proof.
\end{proof}

As a result of Proposition \ref{son1} we have the following well-known
equation \cite[p. 34]{KNB}.

\begin{corollary}
\label{son4}For $k\in\mathbb{N}$ 
\begin{equation}
\sum_{i=1}^{k}\left( -1\right) ^{i+1}\binom{k}{i}H_{i}=\frac{1} {k}.
\label{bph}
\end{equation}
\end{corollary}

\begin{proof}
	Choosing $n=0$\ in Proposition \ref{son1} gives result.
\end{proof}

\begin{remark}
Let us consider the binomial transform \cite{KNB, R2, R3} 
\begin{equation}
a_{k}= \dsum \limits_{i=0}^{k} \binom{k}{i}b_{i}\Leftrightarrow b_{k}= \dsum
\limits_{i=0}^{k} \left( -1\right) ^{k+i}\binom{k}{i}a_{i}.  \label{bp}
\end{equation}
In the light of (\ref{bph}) we can fix $a_{k}=H_{k}$\ and 
\begin{equation*}
b_{k}=\ \left\{ 
\begin{array}{cc}
\frac{\left( -1\right) ^{k+1}}{k} & ,k\in 
\mathbb{Z}
^{+} \\ 
0 & ,k=0%
\end{array}
\right. .
\end{equation*}
Then from (\ref{bp}) we have the well-known formula (see \cite[(1.45)]{HWG}
or \cite[p. 53]{Sch}) 
\begin{equation}
\sum_{i=1}^{k}\left( -1\right) ^{i+1}\binom{k}{i}\frac{1}{i}=H_{k}
\label{hrp}
\end{equation}
where $a_{0}=H_{0}=0=b_{0}.$
\end{remark}

Now we give an identity for the binomial sum of harmonic numbers.

\begin{proposition}
\label{son8}Let $k\geq2$ be an integer. Then 
\begin{equation*}
\dsum \limits_{i=0}^{k} \left( -1\right) ^{i}\binom{k}{i}\left( n+i\right)
H_{n+i}=\frac{\left( k-2\right) !}{\left( n+1\right) ^{\overline{k-1}}}.
\end{equation*}
\end{proposition}

\begin{proof}
	Choosing $f\left(  n\right)  =H_{n}$ in Proposition \ref{teo4} we get%
	\[
	\Delta^{k}\left(  nH_{n}\right)  =n\Delta^{k}H_{n}+k\Delta^{k-1}H_{n+1}.
	\]
	Here by employing (\ref{bphg}) we obtain RHS as%
	\[
	\Delta^{k}\left(  nH_{n}\right)  =n\frac{\left(  -1\right)  ^{k+1}\left(
		k-1\right)  !}{\left(  n+1\right)  ^{\overline{k}}}+k\frac{\left(  -1\right)
		^{k}\left(  k-2\right)  !}{\left(  n+2\right)  ^{\overline{k-1}}}.
	\]
	For the LHS we consider $nH_{n}$ in (\ref{gfark}) and proof follows.
\end{proof}

Now we give a special case of Proposition \ref{son8} which is worth
mentioning. This identity is known \cite{RS}.

\begin{corollary}
For an integer $k\geq2$ 
\begin{equation}
\dsum \limits_{i=1}^{k} \left( -1\right) ^{i}\binom{k}{i}iH_{i}=\frac{1}{%
\left( k-1\right) }.  \label{bih}
\end{equation}
\end{corollary}

\begin{remark}
Considering (\ref{bih}) with (\ref{bp}) we choose $a_{k}=kH_{k}$ and 
\begin{equation*}
b_{k}=\ \left\{ 
\begin{array}{cc}
\frac{\left( -1\right) ^{k}}{\left( k-1\right) }, & k\geq2 \\ 
1, & k=1 \\ 
0, & k=0%
\end{array}
\right.
\end{equation*}
to get 
\begin{equation*}
k\left( H_{k}-1\right) = \dsum \limits_{i=2}^{k} \binom{k}{i}\frac{\left(
-1\right) ^{i}}{\left( i-1\right) }.
\end{equation*}
\end{remark}

\begin{corollary}
We have the following alternate sum of harmonic numbers and binomial
coefficients 
\begin{equation}
\sum_{i=1}^{n-1}\left( -1\right) ^{i+1}\binom{n+1}{i+1}H_{i}=\left\{ 
\begin{array}{ll}
2H_{n}, & n\text{ even} \\ 
0, & n\text{ odd}%
\end{array}
\right.  \label{w}
\end{equation}
{\Huge \ }
\end{corollary}

\begin{proof}
	From Corollary \ref{son4} we have%
	\[
	\sum_{i=1}^{k}\left(  -1\right)  ^{i+1}\binom{k}{i}H_{i}=\frac{1}{k}.
	\]
	Here summing both sides from $1$ to $n$ we get%
	\[%
	{\displaystyle\sum\limits_{k=1}^{n}}
	\sum_{i=1}^{k}\left(  -1\right)  ^{i+1}\binom{k}{i}H_{i}=%
	{\displaystyle\sum\limits_{k=1}^{n}}
	\frac{1}{k}=H_{n}.
	\]
	By changing the order of sums this becomes%
	\[
	\sum_{i=1}^{n}\left(  -1\right)  ^{i+1}H_{i}\sum_{k=i}^{n}\binom{k}{i}=H_{n}.
	\]
	From equation \cite{GKP}%
	\begin{equation}
	\sum_{k=i}^{n}\binom{k}{i}=\binom{n+1}{i+1}\label{bua}%
	\end{equation}
	we find%
	\[
	\sum_{i=1}^{n}\left(  -1\right)  ^{i+1}\binom{n+1}{i+1}H_{i}=H_{n}.
	\]
	
\end{proof}

\begin{remark}
Equation (\ref{w})\ has been already discovered by Wang \cite{W}, but here
we give an alternative proof of it. Besides, we also give one more proof of
it in Appendix by using the equation (1.44) in Gould's book \cite{HWG}.
\end{remark}

\subsection{\textbf{Hyperharmonic numbers via difference operator}}

By taking inspiration from (\ref{phh}), Mez\H{o} \cite{M} defined the
hyperharmonic function:%
\begin{equation}
h_{z}^{\left( w\right) }=\frac{z^{\overline{w}}}{z\Gamma\left( w\right) }%
\left( \psi\left( z+w\right) -\psi\left( w\right) \right)  \label{hhf}
\end{equation}
where $w,$ $z+w\in\mathbb{C}\backslash\left( \mathbb{Z}^{-}\cup\left\{
0\right\} \right) $. Here $\mathbb{Z}^{-}$ denotes the negative integers. In
the light of this definition, Dil \cite{D} gave formulas to evaluate some
special values of $h_{z}^{\left( w\right) }$. However, all those evaluations
are valid under the restriction of $w,$ $z+w\in\mathbb{C}\backslash\left( 
\mathbb{Z}^{-}\cup\left\{ 0\right\} \right) $, i.e. upper index can not be a
negative integer. In this subsection firstly we are going to show a way to
define "negative-ordered hyperharmonic number".

From (\ref{2}) we have the recurrence%
\begin{equation*}
h_{n}^{\left( r\right) }=h_{n-1}^{\left( r\right) }+h_{n}^{\left( r-1\right)
}.
\end{equation*}
To obtain lower-ordered hyperharmonic number in terms of higher-ordered
ones, we can write it as 
\begin{equation}
h_{n}^{\left( r-1\right) }=h_{n}^{\left( r\right) }-h_{n-1}^{\left( r\right)
}.  \label{hhr}
\end{equation}
Setting $r=1$ gives%
\begin{equation*}
h_{n}^{\left( 0\right) }=h_{n}^{\left( 1\right) }-h_{n-1}^{\left( 1\right)
}=H_{n}-H_{n-1}=\frac{1}{n}.
\end{equation*}
For $r=0$ in (\ref{hhr}) gives%
\begin{equation*}
h_{n}^{\left( -1\right) }=h_{n}^{\left( 0\right) }-h_{n-1}^{\left( 0\right)
}=\frac{1}{n}-\frac{1}{n-1}=\frac{-1}{n^{\underline{2}}}.
\end{equation*}
Similarly, for $r=-1$ and $r=-2$ we have%
\begin{equation*}
h_{n}^{\left( -2\right) }=\frac{2}{n^{\underline{3}}}
\end{equation*}
and%
\begin{equation*}
h_{n}^{\left( -3\right) }=\frac{-6}{n^{\underline{4}}}
\end{equation*}
respectively. So by induction on $r$ for $n>1$ we have%
\begin{equation*}
h_{n}^{\left( -r\right) }=\frac{\left( -1\right) ^{r}r!}{n^{\underline{r+1}}}%
.
\end{equation*}
Hence we deduce the following definition.

\begin{definition}
For positive integers $n$ and $r$, the negative-ordered hyperharmonic number 
$h_{n}^{\left( -r\right) }$ is defined by 
\begin{equation*}
h_{n}^{\left( -r\right) }=\left\{ 
\begin{array}{ll}
\frac{\left( -1\right) ^{r}r!}{n^{\underline{r+1}}}, & n>r\geq1 \\ 
0, & r\geq n>1 \\ 
1, & n=1%
\end{array}
\right.
\end{equation*}
\end{definition}

Now we give some results which are obtained using difference operator. We
can consider $h_{n}^{\left( r\right) }$ either as a function of $n$ or $r$,
hence we get the following proposition which gives a representation of
hyperharmonic numbers with binomial coefficients.

\begin{proposition}
\label{one1}Multiple difference of $h_{n}^{\left( r\right) }$ with respect
to $n$ gives 
\begin{equation}
h_{n+k}^{\left( r-k\right) }= \dsum \limits_{i=0}^{k} \left( -1\right) ^{k-i}%
\binom{k}{i}h_{n+i}^{\left( r\right) },  \label{one3}
\end{equation}
where $r\in 
\mathbb{Z}
$ and $k,n\in\mathbb{N}\cup\left\{ 0\right\} $ and with respect to $r$ gives 
\begin{equation*}
h_{n-k}^{\left( r+k\right) }= \dsum \limits_{i=0}^{k} \left( -1\right) ^{k-i}%
\binom{k}{i}h_{n}^{\left( r+i\right) } ,  \label{one10}
\end{equation*}
where $r\in 
\mathbb{Z}
$ and $k,n\in\mathbb{N}$ and $k\leq n.$
\end{proposition}

\begin{proof}
	Proof can be seen from (\ref{gfark}) and multiple difference of $h_{n}%
	^{\left(  r\right)  }$ respect to the variables $n$ and $r$.
\end{proof}

\begin{corollary}
For $k\in\mathbb{N}$ 
\begin{equation*}
h_{n-k}^{\left( k\right) }= \dsum \limits_{i=0}^{k} \left( -1\right) ^{k-i}%
\binom{k}{i}h_{n}^{\left( i\right) },
\end{equation*}
and 
\begin{equation*}
\dsum \limits_{i=0}^{k} \left( -1\right) ^{k-i}\binom{k}{i}h_{k}^{\left(
r+i\right) }=0.
\end{equation*}
\end{corollary}

As an immediate result of (\ref{one3}), we give a new representation for
negative-ordered hyperharmonic numbers.

\begin{corollary}
For $k,n\in\mathbb{N}$, the $\left( n+k\right) $-$th$ hyperharmonic number
of order $-k$ is\ given by 
\begin{equation}
h_{n+k}^{\left( -k\right) }= \dsum \limits_{i=0}^{k} \left( -1\right) ^{k-i}%
\binom{k}{i}\frac{1}{n+i} .  \label{e}
\end{equation}
\end{corollary}

Using the following result one can state lower-ordered hyperharmonic numbers
in terms of higher-ordered hyperharmonic numbers.

\begin{corollary}
We have 
\begin{equation*}
h_{k}^{\left( r-k\right) }= \dsum \limits_{i=1}^{k} \left( -1\right) ^{k-i}%
\binom{k}{i}h_{i}^{\left( r\right) }.
\end{equation*}
\end{corollary}

\begin{corollary}
For non-negative integer $k$ and $n$ we have 
\begin{equation*}
\frac{1}{k+1}= \dsum \limits_{i=0}^{k} \binom{k}{i}h_{i+1}^{\left( -i\right)
},
\end{equation*}
and 
\begin{equation}
H_{n}= \dsum \limits_{i=1}^{n} \binom{n}{i}h_{i}^{\left( 1-i\right) }.
\label{e2}
\end{equation}
\end{corollary}

\begin{proof}
	From (\ref{e}) we have%
	\[
	h_{k+1}^{\left(  -k\right)  }=%
	{\displaystyle\sum\limits_{i=0}^{k}}
	\left(  -1\right)  ^{k-i}\binom{k}{i}\frac{1}{i+1}.
	\]
	In the light of (\ref{bp}) we write%
	\[
	\frac{1}{k+1}=%
	{\displaystyle\sum\limits_{i=0}^{k}}
	\binom{k}{i}h_{i+1}^{\left(  -i\right)  }.
	\]
	For the second equation, we write%
	\begin{align*}%
	{\displaystyle\sum\limits_{k=0}^{n-1}}
	\frac{1}{k+1}  & =%
	{\displaystyle\sum\limits_{k=0}^{n-1}}
	{\displaystyle\sum\limits_{i=0}^{k}}
	\binom{k}{i}h_{i+1}^{\left(  -i\right)  }\\
	& =%
	{\displaystyle\sum\limits_{i=0}^{n-1}}
	{\displaystyle\sum\limits_{k=i}^{n-1}}
	\binom{k}{i}h_{i+1}^{\left(  -i\right)  }.
	\end{align*}
	With the help of (\ref{bua}) we get%
	\[%
	{\displaystyle\sum\limits_{k=0}^{n-1}}
	\frac{1}{k+1}=%
	{\displaystyle\sum\limits_{i=0}^{n-1}}
	\binom{n}{i+1}h_{i+1}^{\left(  -i\right)  }%
	\]
	which completes the proof.
\end{proof}

\begin{remark}
By definition we have 
\begin{equation*}
h_{i+1}^{\left( -i\right) }=\frac{\left( -1\right) ^{i}}{i+1}.
\end{equation*}
Employing this in (\ref{e2}) gives an alternative proof of (\ref{hrp}).
\end{remark}

\begin{remark}
As a result we can state 
\begin{equation*}
H_{k}= \dsum \limits_{i=1}^{k} \left( -1\right) ^{k-i}\binom{k}{i}%
h_{i}^{\left( k+1\right) }.
\end{equation*}
Furthermore we obtain 
\begin{equation*}
\frac{1}{k}= \dsum \limits_{i=1}^{k} \left( -1\right) ^{k-i}\binom{k}{i}%
h_{i}^{\left( k\right) },
\end{equation*}
and from this we have a double sum representation for harmonic number as 
\begin{equation*}
H_{n}= \dsum \limits_{1\leq i\leq k\leq n} \left( -1\right) ^{k-i}\binom{k}{i%
}h_{i}^{\left( k\right) }.
\end{equation*}
\end{remark}

\subsection{\textbf{Fibonacci numbers via difference operator}}

In spite of the fact that Fibonacci numbers originally come from a counting
problem, it is possible to define negative-indexed Fibonacci number as: 
\begin{equation*}
F_{-n}=\left( -1\right) ^{n+1}F_{n}
\end{equation*}
where $n\in\mathbb{N}$. The recurrence relation%
\begin{equation*}
F_{-n}=F_{-n-1}+F_{-n-2}
\end{equation*}
also holds \cite{ber, Dun}.

Now we give binomial representations for Fibonacci and negative-indexed
Fibonacci numbers.

\begin{proposition}
\label{one2} For non-negative integers $k$ and $n$ we have 
\begin{equation*}
F_{n-k}=\sum_{i=0}^{k}\left( -1\right) ^{k-i}\binom{k}{i}F_{n+i} .
\label{son5}
\end{equation*}
\end{proposition}

\begin{proof}
	(\ref{gfark}) gives%
	\[
	\Delta^{k}F_{n}=\sum_{i=0}^{k}\left(  -1\right)  ^{k-i}\binom{k}{i}F_{n+i}.
	\]
	On the other hand we calculate $\Delta^{k}F_{n}=F_{n-k}$, equality of these
	two results completes the proof.
\end{proof}

By setting $n=0$ in Proposition \ref{son5}, we get the following result.

\begin{corollary}
\label{nfib}For $k\in\mathbb{N}$ we have 
\begin{equation}
F_{-k}=\sum_{i=0}^{k}\left( -1\right) ^{k-i}\binom{k}{i}F_{i}.  \label{NF}
\end{equation}
\end{corollary}

RHS of (\ref{NF}) is meaningful for \ $k>0$ therefore we can consider (\ref%
{NF})\ as a definition of negative-indexed Fibonacci number. Calculations
give the first few of them as $F_{-1}=1$, $F_{-2}=-1$, $F_{-3}=2,\ldots$ .

\begin{remark}
The negative-indexed Fibonacci numbers have been already discovered \cite%
{ber, Dun}. Here we give another approach.
\end{remark}

By using (\ref{NF}), we give an alternative verification of the recurrence
relation for negative-indexed Fibonacci number.

\begin{corollary}
\label{son6} For any\ $k\in\mathbb{N}\cup\left\{ 0\right\} $ we have 
\begin{equation*}
F_{-k}=F_{-k-1}+F_{-k-2}.
\end{equation*}
\end{corollary}

\begin{proof}
	Considering Corollary \ref{nfib} we write%
	\begin{align*}
	F_{-k-1}+F_{-k-2}  & =F_{k+2}+\sum_{i=0}^{k+1}\left(  -1\right)  ^{k-i}\left[
	\binom{k+2}{i}-\binom{k+1}{i}\right]  F_{i}\\
	& =F_{k+2}+\sum_{i=1}^{k+1}\left(  -1\right)  ^{k-i}\binom{k+1}{i-1}F_{i}\\
	& =\sum_{i=1}^{k+2}\left(  -1\right)  ^{k-i}\binom{k+1}{i-1}F_{i}%
	\end{align*}
	which can be written as%
	\begin{align*}
	F_{-k-1}+F_{-k-2}  & =\sum_{i=1}^{k+2}\left(  -1\right)  ^{k-i}\left[
	\binom{k}{i-1}+\binom{k}{i-2}\right]  F_{i}\\
	& =\sum_{i=1}^{k+1}\left(  -1\right)  ^{k-i}\binom{k}{i-1}F_{i}+\sum
	_{i=2}^{k+2}\left(  -1\right)  ^{k-i}\binom{k}{i-2}F_{i}\\
	& =\sum_{i=0}^{k}\left(  -1\right)  ^{k-i}\binom{k}{i}\left[  F_{i+2}%
	-F_{i+1}\right] \\
	& =\sum_{i=0}^{k}\left(  -1\right)  ^{k-i}\binom{k}{i}F_{i}\\
	& =F_{-k}.
	\end{align*}
	
\end{proof}

\begin{remark}
Eventually for any $k\in 
\mathbb{Z}
$\ we have 
\begin{equation*}
F_{k}=F_{k-1}+F_{k-2}.
\end{equation*}
\end{remark}

Lastly we prove a correspondence between negative- and positive-indexed
Fibonacci numbers.

\begin{corollary}
The following identity holds: 
\begin{equation*}
F_{-k}=\left( -1\right) ^{k+1}F_{k}.
\end{equation*}
\end{corollary}

\begin{proof}
	Considering (\ref{NF}) we obtain the result,\ by induction on $k$.
\end{proof}

\subsection{\textbf{Hyperbolic functions via difference operator}}

Now we apply the difference operator to the hyperbolic functions.

\begin{proposition}
\label{one11} Effects of $\Delta$ operator on $\sinh x$ and $\cosh x$\
functions are given by 
\begin{equation*}
\dsum \limits_{i=0}^{k} \left( -1\right) ^{k-i}\binom{k}{i}\sinh\left(
x+i\right) =\frac{1} {2e^{x}}\left( 1-\frac{1}{e}\right) ^{k}\left(
e^{2x+k}+\left( -1\right) ^{k+1}\right)
\end{equation*}
and 
\begin{equation*}
\dsum \limits_{i=0}^{k} \left( -1\right) ^{k-i}\binom{k}{i}\cosh\left(
x+i\right) =\frac{1} {2e^{x}}\left( 1-\frac{1}{e}\right) ^{k}\left(
e^{2x+k}+\left( -1\right) ^{k}\right) .
\end{equation*}
\end{proposition}

\begin{proof}
	Proof can be seen from (\ref{gfark}) and multiple differences of $\sinh x$ and
	$\cosh x$ functions.
\end{proof}

\begin{remark}
As a consequence of Proposition \ref{one11} we have closed formulas for the
alternate binomial sums of $\sinh x$ and $\cosh x:$ 
\begin{equation*}
\dsum \limits_{i=0}^{k} \left( -1\right) ^{k-i}\binom{k}{i}\sinh\left(
i\right) =\frac{\left( e-1\right) ^{k}\left( e^{k}+\left( -1\right)
^{k+1}\right) }{2e^{k}}
\end{equation*}
and 
\begin{equation*}
\dsum \limits_{i=0}^{k} \left( -1\right) ^{k-i}\binom{k}{i}\cosh\left(
i\right) =\frac{\left( e-1\right) ^{k}\left( e^{k}+\left( -1\right)
^{k}\right) }{2e^{k}}.
\end{equation*}
\end{remark}

Now we have results on digamma function.

\begin{proposition}
\label{one6} Considering digamma function with difference operator we get 
\begin{equation*}
\frac{\left( k-1\right) !}{x^{\overline{k}}}= \dsum \limits_{i=0}^{k} \left(
-1\right) ^{i}\binom{k}{i}\psi\left( x+i\right) .
\end{equation*}
\end{proposition}

\begin{proof}
	Proof can be seen from (\ref{gfark}) and multiple differences of $\psi\left(  x\right)$.
\end{proof}

\begin{remark}
Considering $k=1$ in Proposition \ref{one6} we give the following well-known
result as a consequence: 
\begin{equation*}
\psi\left( x+1\right) =\psi\left( x\right) +\frac{1}{x}.
\end{equation*}
\end{remark}

\section{\textbf{APPENDIX}}

In Subsection 2.2, we have given the equation (\ref{11}) as the
generalization of the equation (\ref{10}). In a similar fashion one can
obtain the following equations:

\begin{equation}
D\binom{x+n}{n}^{-1}\mid_{x=0}=-H_{n},  \label{A1}
\end{equation}%
\begin{equation}
D\binom{x+r-1+n}{n}^{-1}\mid_{x=0}=-\binom{r-1+n}{n}^{-2}h_{n}^{\left(
r\right) } ,  \label{A2}
\end{equation}%
\begin{equation}
D\binom{x-1}{n}\mid_{x=0}=\left( -1\right) ^{n+1}H_{n} .  \label{A3}
\end{equation}

When these formulas are applied to appropriate sums containing the binomial
coefficients, many equations containing harmonic numbers (and also
generalizations of harmonic numbers) can be obtained. H.W. Gould gave a
list\ of 500 binomial coefficient summations in \cite{HWG}. Considering the
equations (\ref{10}), (\ref{11}), (\ref{A1}), (\ref{A2}) and (\ref{A3})
together with that list we have several identities of harmonic,
hyperharmonic and generalized harmonic numbers. Some of them are known but
most of them are new. The following references also include formulas of
these types: \cite{KNB, CS, C, JS, RS}.

In the following tables, the numbers on the LHS are the numbers of formulas
at the H.W. Gould's book \cite{HWG}; whereas on the RHS are the equations we
have obtained using these formulas.

\begin{remark}
In the following tables, there are some formulas containing hyperharmonic
numbers with non-integer order. We recall that the meanings of these numbers
are given by the hyperharmonic function defined by (\ref{hhf}). For detailed
calculation techniques of such numbers, see \cite{D}.
\end{remark}

\begin{center}
\begin{tabular}{|l|l|}
\hline
(1.23) & $\sum\limits_{k=0}^{\infty}\frac{H_{k}}{2^{k}}=2\ln2$ \\ \hline
(1.41) & $\sum\limits_{k=1}^{n}\left( -1\right) ^{k-1}\binom{n}{k}\frac {1}{k%
}=H_{n}$ \\ \hline
(1.42) & $\sum\limits_{k=1}^{n}\left( -1\right) ^{k-1}\binom{n}{k} H_{k}=%
\frac{1}{n}$ \\ \hline
(1.44) & $\sum\limits_{k=1}^{n}\left( -1\right) ^{k-1}\binom{n+1}{k+1}
H_{k}=H_{n}$ \\ \hline
(2.16) & $\sum\limits_{k=1}^{\infty}\frac{H_{k}}{k!}=e\sum\limits_{k=1}
^{\infty}\frac{\left( -1\right) ^{k-1}}{k!k}$ \\ \hline
(3.2) & $\sum\limits_{k=1}^{n}\binom{r-2+n-k}{n-k}H_{k}=h_{n}^{\left(
r\right) }$ \\ \hline
(3.36) & $2\sum\limits_{k=1}^{2n}\left( -1\right) ^{k}H_{k}=H_{n}$ \\ \hline
(3.95) & $\sum\limits_{k=1}^{n}\left( -1\right) ^{k}\binom{2n-2k}{n-k} 
\binom{2k}{k}\frac{1}{k}=2^{2n}\left\{ h_{n}^{\left( \frac{1}{2}\right) }-%
\binom{n-\frac{1}{2}}{n}H_{n}\right\} $ \\ \hline
(3.100) & $\sum\limits_{k=1}^{n}\left( -1\right) ^{k}\binom{n+k}{2k} \binom{%
2k}{k}\frac{1}{k}=-2H_{n}$ \\ \hline
(3.108) & $\sum\limits_{k=0}^{n}\left\{ \binom{k+m+1}{m}H_{k}+h_{m}^{\left(
k+2\right) }\right\} =\sum_{k=0}^{m}\left\{ \binom{k+n+1}{n}H_{k}
+h_{n}^{\left( k+2\right) }\right\} $ \\ \hline
(4.3) & $\sum\limits_{k=1}^{n}\left( -1\right) ^{k-1}\binom{n}{k}x^{k}
H_{k}=n\left( 1-x\right) ^{n-1}+\sum\limits_{k=1}^{n}\binom{n}{k} \frac{%
\left( x-k\right) ^{k}\left( 1-x+k\right) ^{n-k}}{k}$ \\ \hline
(6.19) & $\sum\limits_{k=0}^{n}\binom{n}{k}\binom{r}{k}h_{n+r}^{\left(
1-k\right) }=H_{n}+H_{r}$ \\ \hline
(6.22) & $\sum\limits_{k=1}^{n}\left( -1\right) ^{k}\binom{2n}{k} \binom{2n-k%
}{n}^{2}\frac{1}{k}=\binom{2n}{n}\left\{ h_{n}^{\left( n+1\right) }-\binom{2n%
}{n}H_{n}\right\} $ \\ \hline
(7.2) & $\sum\limits_{k=1}^{n}\binom{n}{k}\binom{r}{k}H_{k}=\binom{r+n} {n}%
H_{n}-h_{n}^{\left( r+1\right) }$ \\ \hline
(7.9) & $\sum\limits_{k=1}^{n}\left( -1\right) ^{k}\binom{n}{k}\binom{2k} {k}%
^{-1}\frac{2^{2k}}{\left( 2k+1\right) }H_{k}=\binom{2n}{n}^{-1} \frac{2^{2n}%
}{\left( 2n+1\right) }h_{n}^{\left( \frac{1}{2}\right) } $ \\ \hline
(7.13) & $\sum\limits_{k=1}^{n}\left( -1\right) ^{k-1}\binom{n}{k} \binom{%
2n-k}{n-k}H_{k}=\sum\limits_{k=1}^{n}\binom{n}{k}^{2}\frac{1}{k} $ \\ \hline
\end{tabular}

\begin{tabular}{|l|l|}
\hline
(7.15) & $\sum\limits_{k=1}^{n}\binom{n}{k}\binom{r}{k}\left\{ \left(
H_{k}\right) ^{2}+H_{k}^{\left( 2\right) }\right\} =\binom{r+n}{n}\left\{
H_{n}^{\left( 2\right) }-H_{n+r}^{\left( 2\right) }+H_{r}^{\left( 2\right)
}\right\} $ \\ 
& $\ +\left( \binom{r+n}{n}H_{n}-h_{n}^{\left( r+1\right) }\right) \left\{
H_{n}-H_{n+r}+H_{r}\right\} $ \\ \hline
(12.9) & $\sum\limits_{k=0}^{n}\binom{n}{k}\binom{x+k}{k}^{-1}\binom {x+k+1+n%
}{n}^{-1}$ \\ 
& $\ \times\left\{ \frac{x+2k+1}{x+k+1}\left( H_{k}-\binom{x+k+1+n}{n}
^{-1}h_{n}^{\left( x+k+2\right) }\right) -\frac{k}{\left( x+k+1\right) ^{2}}%
\right\} =0$ \\ \hline
(12.9) & $\sum\limits_{k=0}^{n}\binom{n}{k}\binom{y+k}{k}\binom{y+k+1+n} {n}%
^{-1}$ \\ 
& $\ \times\left\{ \frac{y+2k+1}{y+k+1}\left( H_{k}+\binom{y+k+1+n}{n}
^{-1}h_{n}^{\left( y+k+2\right) }\right) -\frac{k}{\left( y+k+1\right) ^{2}}%
\right\} =H_{n}$ \\ \hline
(Z.58) & $H_{2n}=2^{2n}\binom{2n}{n}^{-1}\left\{ \binom{n-\frac{1}{2}} {n}%
H_{n}+h_{n}^{\left( \frac{1}{2}\right) }\right\} $ \\ \hline
\end{tabular}
\end{center}

The next table contains generalized version of the equations in the first
table, by substitution $\binom {x+k}{k}$ with $\binom{x+r-1+k}{k}$ where $%
r>1.$

\begin{center}
\begin{tabular}{|l|l|}
\hline
(1.23) & $\sum\limits_{k=0}^{\infty }\frac{h_{k}^{\left( r\right) }}{2^{k}}%
=2^{r}\ln 2$ \\ \hline
(1.41) & $\sum\limits_{k=1}^{n}\left( -1\right) ^{k-1}\binom{n}{k}\frac{k}{%
\left( k+r-1\right) ^{2}}=\binom{r-1+n}{n}^{-2}h_{n}^{\left( r\right) }$ \\ 
\hline
(1.42) & $\sum\limits_{k=1}^{n}\left( -1\right) ^{k-1}\binom{n}{k}\binom{%
k+r-1}{k}^{-2}h_{k}^{\left( r\right) }=\frac{1}{n+r-1}$ \\ \hline
(1.44) & $\sum\limits_{k=1}^{n}\left( -1\right) ^{k-1}\binom{n+1}{k+1}\binom{%
r-1+k}{k}^{-2}h_{k}^{\left( r\right) }=\sum\limits_{k=1}^{n}\frac{k}{\left(
k+r-1\right) ^{2}}$ \\ \hline
(2.16) & $\sum\limits_{k=1}^{\infty }\binom{r-1+k}{k}^{-2}\frac{%
h_{k}^{\left( r\right) }}{k!}=e\sum\limits_{k=0}^{\infty }\frac{\left(
-1\right) ^{k}}{\left( r+k\right) ^{2}k!}$ \\ \hline
(3.2) & $\sum\limits_{k=1}^{n}\binom{r+n-k}{n-k}h_{k}^{\left( r\right)
}=h_{n}^{\left( r+n+1\right) }$ \\ \hline
(3.36) & $2\sum\limits_{k=1}^{2n}\left( -1\right) ^{k}\binom{r-1+2n-k}{2n-k}%
h_{k}^{\left( r\right) }=h_{n}^{\left( r\right) }$ \\ \hline
(3.95) & $\sum\limits_{k=1}^{n}\left( -1\right) ^{k}\binom{2n-2k}{n-k}\binom{%
2k}{k}\frac{k}{\left( r-1+k\right) ^{2}}$ \\ 
& $=2^{2n}\binom{r-1+n}{n}^{-1}\left\{ h_{n}^{\left( r-\frac{1}{2}\right) }-%
\binom{r-1+n}{n}^{-1}\binom{r-\frac{3}{2}+n}{n}h_{n}^{\left( r\right)
}\right\} $ \\ \hline
(3.100) & $\sum\limits_{k=1}^{n}\left( -1\right) ^{k}\binom{n+k}{2k}\binom{2k%
}{k}\frac{k}{\left( r-1+k\right) ^{2}}$ \\ 
& $=\left( -1\right) ^{n}\binom{r-1+n}{n}^{-1}\left\{ h_{n}^{\left(
r-n-1\right) }-\binom{r-2}{n}\binom{r-1+n}{n}^{-1}h_{n}^{\left( r\right)
}\right\} $ \\ \hline
(3.108) & $\sum\limits_{k=0}^{n}\left\{ \binom{k+r+m}{m}h_{k}^{\left(
r\right) }+\binom{k+r-1}{k}h_{m}^{\left( k+r+1\right) }\right\} $ \\ 
& $=\sum\limits_{k=0}^{m}\left\{ \binom{k+r+n}{n}h_{k}^{\left( r\right) }+%
\binom{k+r-1}{k}h_{n}^{\left( k+r+1\right) }\right\} $ \\ \hline
\end{tabular}

\begin{tabular}{|l|l|}
\hline
(4.3) & $\sum\limits_{k=1}^{n}\left( -1\right) ^{k}\binom{n}{k}\left(
x+r-1\right) ^{k-1}\binom{r-1+k}{k}^{-1}\left\{ k-\left( x+r-1\right) \binom{%
r-1+k}{k}^{-1}h_{k}^{\left( r\right) }\right\} $ \\ 
& $=\sum\limits_{k=1}^{n}\binom{n}{k}\left( x-k\right) ^{k}\left(
1-x+k\right) ^{n-k}\frac{k}{\left( r-1+k\right) ^{2}}$ \\ \hline
(6.19) & $\sum\limits_{k=0}^{n}\binom{n}{k}\binom{r}{k}h_{n+r}^{\left(
j-k\right) }=h_{r}^{\left( j\right) }\binom{n+j-1}{n}+\binom{r+j-1}{r}%
h_{n}^{\left( j\right) }$ \\ \hline
(6.22) & $\sum\limits_{k=1}^{n}\left( -1\right) ^{k}\binom{2n}{k}\binom{2n-k%
}{n}^{2}\frac{k}{\left( r-1+k\right) ^{2}}$ \\ 
& $=\binom{2n}{n}\binom{r-1+n}{n}^{-1}\left\{ h_{n}^{\left( n+r\right) }-%
\binom{2n+r-1}{n}\binom{r-1+n}{n}^{-1}h_{n}^{\left( r\right) }\right\} $ \\ 
\hline
(7.2) & $\sum\limits_{k=1}^{n}\binom{n}{k}\binom{m}{k}\binom{r-1+k}{k}%
^{-2}h_{k}^{\left( r\right) }$ \\ 
& $=\binom{r-1+n}{n}^{-1}\left\{ \binom{r-1+m+n}{n}\binom{r-1+n}{n}%
^{-1}h_{n}^{\left( r\right) }-h_{n}^{\left( m+r\right) }\right\} $ \\ \hline
(7.9) & $\sum\limits_{k=1}^{n}\left( -1\right) ^{k}\binom{n}{k}\binom{2k}{k}%
^{-1}\frac{2^{2k}}{\left( 2k+1\right) }h_{k}^{\left( r\right) }=\binom{2n}{n}%
^{-1}\frac{2^{2n}}{\left( 2n+1\right) }h_{n}^{\left( r-\frac{1}{2}\right) }$
\\ \hline
(7.13) & $\sum\limits_{k=1}^{n}\binom{n}{k}\binom{-n-1}{n-k}\binom{j-1+k}{k}%
^{-2}h_{k}^{\left( j\right) }=\left( -1\right) ^{n+1}\sum\limits_{k=1}^{n}%
\binom{n}{k}^{2}\frac{k}{\left( k+j-1\right) ^{2}}$ \\ \hline
(7.29) & $\frac{\partial }{\partial j}S_{j+r-1}^{n}\mid _{j=0}=\frac{%
\partial }{\partial j}F\left( -2n,\frac{1}{2},n+j+r;4\right) \mid _{j=0}$ \\ 
& $\ \ \ \ \ \ \ \ \ \ \ \ \ \ \ \ \ \ \ =\sum\limits_{k=0}^{2n}\left(
-1\right) ^{k+1}\binom{2n}{k}\binom{2k}{k}\binom{n+r-1+k}{k}%
^{-2}h_{k}^{\left( n+r\right) }$ \\ \hline
(7.30) & $\frac{\partial }{\partial j}R_{j+r-1}^{n}\mid _{j=0}=\frac{%
\partial }{\partial j}F\left( -2n-1,\frac{1}{2},n+j+r;4\right) \mid _{j=0}$
\\ 
& $\ \ \ \ \ \ \ \ \ \ \ \ \ \ \ \ \ \ \ =\sum\limits_{k=0}^{2n+1}\left(
-1\right) ^{k+1}\binom{2n+1}{k}\binom{2k}{k}\binom{n+r-1+k}{k}%
^{-2}h_{k}^{\left( n+r\right) }$ \\ \hline
(12.9) & $\sum\limits_{k=0}^{n}\binom{n}{k}\binom{x+k}{k}^{-1}\binom{x+r+k+n%
}{n}^{-1}\binom{r-1+k}{k}\left\{ \binom{r-1+k}{k}^{-1}\frac{x+r+2k}{x+r+k}%
h_{k}^{\left( r\right) }\right. $ \\ 
& $\left. -\binom{x+r+k+n}{n}^{-1}\frac{x+r+2k}{x+r+k}h_{n}^{\left(
x+k+r+1\right) }-\frac{k}{\left( x+r+k\right) ^{2}}\right\} =0$ \\ \hline
(12.9) & $\sum\limits_{k=0}^{n}\binom{n}{k}\binom{y+k}{k}\binom{r+y+k+n}{n}%
^{-1}\binom{r-1+k}{k}^{-1}\left\{ \binom{r-1+k}{k}^{-1}\frac{r+y+2k}{r+y+k}%
h_{k}^{\left( r\right) }\right. $ \\ 
& $\left. \text{ \ \ }+\binom{r+y+k+n}{n}^{-1}\frac{r+y+2k}{r+y+k}%
h_{n}^{\left( y+k+r+1\right) }+\frac{k}{\left( r+y+k\right) ^{2}}\right\}
=h_{n}^{\left( r\right) }\binom{r-1+n}{n}^{-2}$ \\ \hline
(Z.58) & $\binom{2n}{n}h_{2n}^{\left( 2r-1\right) }=2^{2n}\left\{ \binom{n+r-%
\frac{3}{2}}{n}h_{n}^{\left( r\right) }+\binom{n+r-1}{n}h_{n}^{\left( r-%
\frac{1}{2}\right) }\right\} $ \\ \hline
\end{tabular}
\bigskip
\end{center}

\end{document}